\documentclass{amsart}

\usepackage{amsmath,amssymb,amsfonts,verbatim}

\usepackage{color}
\usepackage{graphicx}
\usepackage{tikz}
\usepackage{tkz-euclide}
\usepackage{subcaption}
\usepackage{hyperref}

\usepackage{svg}

\usepackage{geometry}
\newcommand{\R}{\mathcal{R}}

\newcommand{\F}{\mathcal{F}}
\newcommand{\U}{\mathcal{U}}
\newcommand{\N}{\mathcal{N}}
\newcommand{\PA}{\mathcal{P}}
\newtheorem{theorem}{Theorem}[section]
\newtheorem{lemma}[theorem]{Lemma}

\theoremstyle{definition}

\theoremstyle{remark}
\newtheorem{remark}[theorem]{Remark}

\numberwithin{equation}{section}

\begin{document}

\title{Remarks on differential inclusion limits of stochastic approximation}

\author{Vivek S.\ Borkar}
\address{Department of Electrical Engineering, Indian Institute of Technology Bombay, Powai, Mumbai 400076, India}
\email{borkar.vs@gmail.com}
\thanks{The first author was supported in part by a S.\ S.\ Bhatnagar Fellowship from CSIR, Government of India.}

\author{Dhruv A.\ Shah}
\address{Department of Electrical Engineering, Indian Institute of Technology Bombay, Powai, Mumbai 400076, India}
\email{190020039@iitb.ac.in}

\begin{abstract}
For stochastic approximation algorithms with discontinuous dynamics, it is shown that under suitable distributional assumptions, the interpolated iterates track a Fillipov solution of the limiting differential inclusion. In addition, we give an alternative control theoretic approach to recent results of \cite{Bianchi} on certain limiting empirical measures associated with the iteration.
\end{abstract}

\maketitle

\section{Introduction}

One popular approach to the analysis of stochastic approximation algorithms is to look at their limiting differential equations and interpret the time-asymptotic behavior of the algorithm in terms of that of the differential equation. There are, however, many situations wherein the limiting dynamics is not a differential equation, but a differential inclusion. One important situation where this happens is when the update equation for the algorithm involves a discontinuous function. This leads to a differential equation with a discontinuous right hand side. This is ill-posed in the classical sense. So one resorts to different solution concepts \cite{Cortes} which lead to differential inclusions. There is a considerable body of work in this direction (some of which we outline below). Our contribution here is twofold: First, to refine the set-valued map in the limiting differential equation from a Krasovskii map to a Filippov map under suitable additional conditions, and second, to provide an alternative control theoretic approach to some important recent contributions of \cite{Bianchi}, with some new insights.

We recall the basic facts about stochastic approximation iterates and their differential equation, resp., differential inclusion limits in the next section. Section 3 describes our main results. 

\section{Stochastic approximation: preliminaries}

Here we briefly recall the relevant aspects of the stochastic approximation algorithm. Introduced by Robbins and Monro in \cite{RM}, it is a scheme for finding roots of a nonlinear map $h: \R^d\mapsto \R^d$ given its noisy measurements. Specifically, it is the iteration
\begin{equation}
x(n+1) = x(n) + a(n)(h(x(n)) + M(n+1)), \ n \geq 0, \label{RM}
\end{equation}
where $\{M(n)\}$ is a martingale difference sequence with respect to the increasing $\sigma$-fields 
$$\F_n := \sigma(x(0), M(m), m \leq n), n \geq 0, $$ 
representing measurement noise, and $\{a(n)\}$ is a stepsize sequence satisfying
\begin{equation}
a(n) \geq 0 \ \forall \ n \geq 0, \ \sum_na(n) = \infty, \ \sum_na(n)^2 < \infty. \label{step}
\end{equation}
Typical assumptions imposed are: $h$ is Lipschitz and $\{M(n)\}$ satisfies, for some $K > 0$, 
\begin{equation}
\sup_nE\left[\|M(n)\|^4\right] < \infty \ \ \mbox{and} \  \ E\left[\|M(n+1)\|^2 | \F_n\right] \leq K\left(1 + \|x(n)\|^2\right). \label{mgbound}
\end{equation}
While the initial contributions analyzed (\ref{RM}) using probabilistic techniques, another approach developed since the 70's \cite{DerFra, Meerkov, Ljung, Ben1, Ben2} treats (\ref{RM}) as a noisy Euler scheme for the ODE (for `\textit{Ordinary Differential Equation}')
\begin{equation}
\dot{x}(t) = h(x(t)). \label{ode}
\end{equation}
Under suitable conditions, one can show that $\{x(n)\}$ a.s.\  ($:=$ almost surely with respect to the underlying probability   measure) tracks the asymptotic behavior of (\ref{ode}) as $n\uparrow\infty$. The argument goes as follows. (See \cite{Borkar}, Chapter 2, for details). Define the `algorithmic time scale' $t(n), n \geq 0,$ by: $t(0) := 0, t(n) := \sum_{m=0}^{n-1}a(m)$. Then by (\ref{step}), $t(n) \uparrow \infty$. Define $\bar{x}(t), t \geq 0$, by: $\bar{x}(t(n)) := x(n) \ \forall n$ with linear interpolation on each interval $[t(n),t(n+1)]$, so that it is a continuous piecewise linear curve. Fix $T > 0$ and for $n\geq 0$, define $m(n) := \min\{k: t(k) - t(n) \geq t(n)+T\}$. On the interval $[t(n), t(m(n))]$, define the ODE
$$\dot{y}^n(t) = h(y^n(t)), \ t \in [t(n),t(m(n))], \ y^n(t(n)) = x(n).$$
If one can establish a.s.\ boundedness of the iterates, i.e., $\sup_n\|x(n)\|<\infty$ a.s.\ (this usually needs a separate proof), then (\ref{mgbound}) and the square-summability of $\{a(n)\}$ ensures that the $\{\F_n\}$-martingale $\sum_{m=1}^{n-1}a(m)M(m+1)$ converges a.s.\ (Proposition VII.2.3(c), p.\ 149,  of \cite{Neveu}). This and the fact $a(n) \to 0$ implied by (\ref{step}) together imply, via the Gronwall inequality, that 
\begin{equation}
\lim_{n\uparrow\infty}\max_{t\in[t(n),t(m(n))]}\|\bar{x}(t)-y^n(t)\|\to 0 \ \mbox{a.s.} \label{track}
\end{equation}
Here the martingale convergence and vanishing stepsize lead to the vanishing of errors due to noise and discretization, respectively. A similar argument works for $T < 0$. This in turn leads to the following characterization of the asymptotic behavior of $\{x(n)\}$ due to Benaim \cite{Ben1,Ben2}.

\begin{theorem}\label{Benaim}
Almost surely, $x(n)\to$ a nonempty compact connected internally chain transitive invariant set of (\ref{ode}). \end{theorem}

See \textit{ibid.} or Chapter 2 of \cite{Borkar} for details. A more general situation is that of a \textit{stochastic recursive inclusion}, i.e., the iteration
\begin{equation}
x(n+1) = x(n) + a(n)(z(n) + M(n+1)), \ z(n) \in H(x(n)) \ n \geq 0, \label{RM2}
\end{equation}
where $H: x\in\R^d \mapsto H(x)\subset \R^d$ is a set-valued map and $\{M(n)\}$ is as above. We assume that $x \mapsto H(x)$ is a nonempty compact and convex valued map which is upper semicontinuous (i.e., its graph $\{(x,z): z\in H(x)\}$ is closed) and satisfies the linear growth condition
$$\max_{z \in H(x)}\|z\| \  \leq  \ K'(1+\|x\|)$$
for some $K'>0$. Then one looks at a \textit{differential inclusion} limit instead of an o.d.e. This differential inclusion is
\begin{equation}
\dot{x}(t) \in H(x(t)). \label{DiffIncl}
\end{equation}
Following the pioneering work of Benaim, Hofbauer and Sorin \cite{Benaim1, Benaim2}, which, among other things, establish a counterpart of Theorem \ref{Benaim} for this case, this iteration has been extensively studied in literature, some of it motivated by applications to reinforcement learning \cite{Bianchi, Faure1, Faure2, Ram1, Ram2, Yaji1, Yaji2, Yaji3}. The aforementioned extension of Theorem \ref{Benaim}  is as follows.

 \begin{theorem}\label{Benaim2}
Almost surely, $\bar{x}(\cdot)$ asymptotically tracks a solution of (\ref{DiffIncl}) in the sense that (\ref{track}) holds with $y^n(\cdot) :=$ a solution of (\ref{DiffIncl}) on $[t(n), t(m(n))]$  with $y^n(t(n)) = x(n) \ \forall n$. Furthermore, $x(n)\to$ a nonempty compact connected internally chain transitive invariant set of (\ref{ode}). \end{theorem}

Here, an invariant set is a set of points $x$ such that \textit{some} solution of (\ref{DiffIncl}) passing through $x$ remains in the set for all time. 

A special case of interest is (\ref{ode}) when the map $h$ is discontinuous and the usual theory for well-posedness of (\ref{ode}) does not apply. The standard approach (see \cite{Yin} and its references) has been to  treat (\ref{ode}) as a special case of (\ref{DiffIncl}) by setting $z(n) = h(x(n))$ and 
$$H(x) = K_h(x) := \cap_{\delta>0}\overline{co}\left(\{h(y) : \|y-x\|<\delta\}\right).$$
Any trajectory thereof is known as a Krasovskii solution to (\ref{ode}) \cite{Krasov}, one of the many solution concepts for differential equations driven by discontinuous vector fields. (See \cite{Cortes} for a recent survey of the various solution concepts and their inter-relationships.)

\section{Main results}

We assume throughout that
\begin{equation}
\sup_n\|x(n)\| < \infty \ \mbox{a.s.}  \label{stable}
\end{equation}
We shall denote by $\PA(S)$ the Polish space of probability measures on a Polish space $S$ with the Prohorov topology. $C_b(S)$ will denote the space of bounded continuous functions $S \mapsto \R$.

We present two results in this section. The first is the observation that if the laws of $(x(n), M(n+1))$ are Lebesgue-continuous, then we may replace the commonly used Krasovskii solution by the more restrictive Filippov solution \cite{Filippov} for the limiting differential equation. This is facilitated by a remarkable result from \cite{Buckdahn}. 

The second result is in the spirit of another remarkable piece of work \cite{Bianchi} that characterizes the limiting behavior of suitably averaged `occupation measures' $\mu(dxdz|t)$ defined below, as $n\to\infty$. While \cite{Bianchi} uses classical analytic tools such as Young measures and concepts from topological dynamics, we take a control theoretic view based on relaxed controls and Stockbridge's extension of Echeverria's theorem \cite{Echeverria} to controlled martingale problems \cite{Stockbridge}. (See \cite{Bhatt} for a further extension and \cite{Ari}, Chapter 6, for a detailed exposition.) In addition to recovering variants of the results of \cite{Bianchi}, this also gives some additional insights.

\subsection{Filippov solution}

Let $\N$ denote the collection of Lebesgue-null sets in $\R^d$. A Fillipov solution to (\ref{ode}) replaces $K_h(x)$ above by
$$F_h(x) :=  \cap_{N\in\N}\cap_{\delta>0}\overline{co}\left(\{h(y) : \|y-x\|<
\delta\}\backslash N\right).$$
That is, we consider the differential inclusion limit given by
\begin{equation}
\dot{x}(t) \in F_h(x(t)). \label{Filippov}
\end{equation}
The elimination of sets of zero Lebesgue measure does matter. Consider, e.g., the following example.\\

\noindent \textbf{Example 1:} Consider the two dimensional case where $h: \R^2 \mapsto \R^2$ is given by
\begin{eqnarray*}
h(x,y) &=& [1,-1], \ \ \ \ y > 0,\\
&=& [1,1], \ \ \ \ \ \ y < 0,\\
&=& [-1,0], \ \ \ \  y=0.
\end{eqnarray*}
One can easily see that the Filippov solution will be
\begin{eqnarray*}
\dot{x}(t) &=& \ 1, \ \ \ \ \ \ \ \ \  \\
\dot{y}(t) &=& -1, \ \ \ \ \ \ \ \ y>0, \\
&=& \ 1, \ \ \ \ \ \ \ \ \ y< 0, \\
&=& \  0, \ \ \ \ \ \ \ \ \  y=0.
\end{eqnarray*}
In particular, letting $[x,y]$ denote a generic vector in $\R^2$, we have for $y(t) = 0$, $\frac{d}{dt}[x(t),y(t)] =[1,0] \neq [-1,0]$. A Krasovskii solution would only tell us that for $y(t) = 0$, $\frac{d}{dt}[x(t),y(t)] \in \overline{co}\Big([1,-1], [1,1], [-1,0]\Big)$.\\

Assume that:\\

\noindent \textbf{$(\dagger)$}  The laws of $(x(n), M(n+1))$ are absolutely continuous with respect to the Lebesgue measure for all $n$.\\

This holds, e.g., if the law of $x(0)$ and regular conditional law of $M(n+1)$ given $x(n)$ is absolutely continuous w.r.t.\ the Lebesgue measure for all $n$ (`a.s.' in the latter case). Our first main result is:

\begin{theorem}\label{Fil} The $\{x(n)\}$ generated by (\ref{RM}) with a discontinuous $h$ will  asymptotically track a Filippov solution of (\ref{Filippov}), a.s. \end{theorem}

\begin{proof} We use the following result from \cite{Buckdahn}.

\begin{lemma} Let $f : \R^d \mapsto \R^d$ be measurable and locally bounded. Then,
\begin{enumerate}

\item for  a.e.\ ($:=$ almost everywhere with respect to the Lebesgue measure) $x$, $f(x) \in F_f(x)$ and $F_f(\cdot)$ is the smallest closed convex valued upper semicontinuous set valued map for which this holds, and,

\item there exists a measurable $\tilde{f} = f$ a.e.\ such that $F_f(x) = K_{\tilde{f}}(x) \ \forall x\in\R^d$.
\end{enumerate}
\end{lemma}

See Prop.\ 2, pp.\ 231-234, \cite{Buckdahn}, for a proof. In fact, it is easy to see that if one  considers the convolutions $f^\delta := f*\varphi^\delta$ where $\varphi^\delta, \delta > 0,$ are smooth approximations to the Dirac measure at the origin that are supported in the ball of radius $\delta > 0$ centered at the origin, then $f = \lim_{\delta\downarrow 0}f^\delta$ at all Lebesgue points of $f$ \cite{Jost}, and therefore $f(x) \in F_f(x)$ at all Lebesgue points of $f$. 

By the above lemma, there exists a measurable $\tilde{h} : \R^d \mapsto \R^d$ such that $\tilde{h} = h$ a.e.\ and $F_h(\cdot) = K_{\tilde{h}}(\cdot)$. Thus the iteration
$$\tilde{x}(n+1) = \tilde{x}(n) + a(n)(\tilde{h}(x(n)) + M(n+1)), \ n \geq 0,$$
with $\tilde{x}(0) = x(0)$ a.s., can be shown to satisfy $\tilde{x}(n) = x(n)$ a.s.\ by induction, in view of \textbf{($\dagger$)}. But by standard theory for stochastic recursive inclusions (see, e.g., \cite{Benaim1} or Chapter 5 of \cite{Borkar}), $\{\tilde{x}(n)\}$ asymptotically track a.s.\ the differential inclusion  
$$\dot{x}(t) \in K_{\tilde{h}}(x(t)) = F_{h}(x(t)),$$
in view of the foregoing. Then so does $\{x(n)\}$. The claim follows. 
\end{proof}

\subsection{Limiting empirical measures}

In this section we take a control theoretic view of (\ref{DiffIncl}). Let $f_i: \R^d \mapsto \R, i \geq 1,$ be a countable family of bounded twice continuously differentiable functions that has bounded first and second partial derivatives, chosen such that it forms a convergence determining class for $\PA(\R^d)$.  Note that for each $i\geq 1$,
$$\xi_i(n) := \sum_{m=0}^{n-1}a(m)\langle\nabla f_i, M(m+1)\rangle, \ n \geq 1,$$
is a square-integrable martingale w.r.t.\ $\{\F_n\}$.

\begin{lemma}\label{mg}  $\xi_i(n)$ converges a.s.\ as $n\to\infty, \ \forall i$. \end{lemma}

\begin{proof} By (\ref{stable}), the quadratic variation process $\{\langle \xi\rangle_n\}$ of $\{\xi_i(n)\}$ satisfies 
$$\langle\xi_i\rangle_n \leq K_i\sum_{m=0}^{n-1}a(m)^2E\left[\|M(m+1)\|^2|\F_m\right] \leq K_iK\sum_na(n)^2(1 + \|x(n)\|^2) < \infty \ \mbox{a.s.}$$
for a suitable bound $K_i$ on $\|\nabla f_i(\cdot)\|$ and $K$ as in (\ref{mgbound}). By Proposition VII.2.3(c), p.\ 149,  of \cite{Neveu}, $\xi_i(n)$ converges a.s.
\end{proof}

Likewise, we have:
\begin{lemma}\label{extra}  $\sum_na(n)^2\|M(n+1)\|^2 < \infty$ a.s. \end{lemma}
 
\begin{proof} Let
$$Z(n) := \sum_{m=0}^{n-1}a(n)^2\left(\|M(m+1)\|^2 - E\left[\|M(m+1)\|^2|\F_m\right]\right), \ n \geq 1.$$
By our assumption (\ref{mgbound}), $\sup_nE\left\|M(n)\|^4\right]<\infty.$ From this and \eqref{step}, it follows that  the quadratic variation process $\langle Z\rangle_n$ of $\{Z_n\}$ is uniformly bounded in mean square and therefore, is bounded a.s.  In turn,  this and  Proposition VII.2.3(c), p.\ 149,  of \cite{Neveu}, imply that $Z(n)$ converges a.s.\ as $n\uparrow\infty$. Since
$$ \sum_{m=0}^{n-1}a(n)^2E\left[\|M(m+1)\|^2|\F_m\right] \leq  K\sum_{m=0}^{n-1}a(n)^2\left(1 + \|x(m)\|^2\right) < \infty$$
a.s.\ (where $K > 0$ is as in (\ref{mgbound})), the claim follows. \end{proof}

Next, let $B, D$ be closed convex subsets of $\R^d$. We define $\U :=$ the space of measurable maps $[0,\infty) \mapsto \PA(B\times D)$  with the coarsest topology that renders continuous the maps
$$\mu(dxdy|\cdot) \in \U \mapsto \int_s^tg(y)\int_{B\times D}f(x,u)\mu(dxdu|y)dy, \ f\in C(B\times D), g\in L_2[s,t], t > s \geq 0.$$
This is a compact  metrizable topology, hence $\U$ is Polish (See, e.g., pp.\ 71-73 of \cite{Borkar}). In fact this is a special case of  the standard topology for relaxed (or `chattering') controls introduced by L.\ C.\ Young that later led to the more general notion of `Young measures' in calculus of variations \cite{Young}. 

Let $\Phi :=$ a zero probability set outside which (\ref{stable}) holds and the conclusions of Lemmas \ref{mg} and \ref{extra} hold. Define a $\PA((\R^d)^2)$-valued process $\mu(dxdz|t), t \in [0,\infty)$, by
$$\mu(dxdz|t) := \delta_{(x(n),z(n))}(dxdz),  \ t(n) \leq t < t(n+1), n \geq 0,$$
where $\{t(n)\}$ are defined as in the preceding section and $\delta_{(x',z')}(dxdz)$ denotes the Dirac measure at $(x',z')$. 
 Fix a sample point in $\Phi^c$. Choose $B, D \subset \R^d$ such that $(x(n), h(x(n))) \in B\times D \ \forall n$. This is possible by (\ref{stable}) and the linear growth condition on $F_h(\cdot)$. Define $\check{\mu}(dxdz|\cdot) \in \U$ by:
$$\int f(x,z)\check{\mu}(dxdz|t) := \frac{1}{t}\int_0^t\int f(x,z)\mu(dxdz|s)ds$$
for twice continuously differentiable $f \in C_b(B\times D)$ with bounded first and second partial derivatives. In particular,
$$\frac{1}{t(n)}\int_0^{t(n)}\int \langle\nabla f_i(x), z\rangle \mu(dxdz|t)dt = \frac{\sum_{k=0}^{n-1}a(k)\langle\nabla f_i(x(k)), z(k)\rangle}{\sum_{k=0}^{n-1}a(k)}.$$
Note that we work with a fixed sample path and therefore the sets $B,D$, though sample path dependent, are fixed. 

\begin{theorem} For the chosen sample path, any limit point of $\mu(dxdz|t + \cdot)$ in $\U$ as $t\uparrow\infty$ is of the form $\mu^*(dxdz)$ such that,

\begin{enumerate}

\item $\mu^*$ is supported on the graph of $K_h(\cdot)$ and if \textbf{$(\dagger)$} holds, on the graph of $F_h(\cdot)$, and,

\item there exists a stationary process $(X(\cdot), Z(\cdot))$ satisfying $\dot{X}(t) = Z(t)$ for which the marginal law of $(X(t), Z(t))$ is $\mu^*$ for all $t$. Furthermore, one may view the control $Z(\cdot)$ as  corresponding to a stationary Markov relaxed control $v(X(\cdot))$ for some measurable $v :x \in \R^d \mapsto \mathcal{P}(K_h(x))$, ($v :x \in \R^d \mapsto \mathcal{P}(F_h(x))$ if \textbf{$(\dagger)$} holds), i.e.,  $v(X_t)$ is the law of $Z_t$ for all $t\geq 0$.
\end{enumerate}
\end{theorem}

\begin{proof} Recall from the preceding section the definitions of $t(n), n \geq 0,$ and $\bar{x}(\cdot)$. Since $t(n)\to\infty$ and $t(n+1) - t(n) \to 0$, it suffices to consider a subsequence along $\{t(n)\}$ as $n\uparrow\infty$. For $i\geq 1$,
\begin{eqnarray*}
f_i(\bar{x}(t(n)) - f_i(\bar{x}(0)) &=& \sum_{k=0}^{n-1}(f_i(\bar{x}(t(k+1)) - f_i(\bar{x}(t(k)))) \\
&=& \sum_{k=0}^{n-1}a(k)\langle\nabla f_i(x(k)), z(k)\rangle +  \sum_{k=0}^{n-1}a(k)\langle \nabla f_i(x(k)), M(k+1)\rangle + \sum_{k=0}^{n-1}\zeta(k)
\end{eqnarray*}
by Taylor formula, where $\|\zeta(k)\| \leq Ca(k)^2\left(1+\|M(k+1)\|^2\right)$ for some constant $C > 0$. By our choice of the sample path, the conclusion of Lemma \ref{mg} applies and the second term is bounded. Likewise by Lemma \ref{extra}, so is the third term. Then dividing both sides by $t(n)$ and  letting $n\uparrow\infty$,  we have $t(n)\uparrow\infty$ and therefore
$$\frac{1}{n}\sum_{k=0}^{n-1}a(k)\langle\nabla f_i(x(k)), z(k)\rangle = \frac{1}{t(n)}\int_0^{t(n)}\int \langle\nabla f_i(x), z\rangle d\mu(dxdz|t)dt \to 0.$$
Suppose $\mu^*(dxdz)$ is a limit point of 
$\check{\mu}(dxdz| t)$ as $t\to\infty$. Then by passing to the limit along an appropriate subsequence, we have 
$$\int\langle\nabla f_i(x), z\rangle\mu^*(dxdz) = 0  \ \forall i.$$
Then the first claim follows from Theorem \ref{Benaim2} and Theorem \ref{Fil}.  By Theorem 4.7 of \cite{Stockbridge}, there exists a stationary process $(X(\cdot), Z(\cdot))$ such that $\dot{X}(t) = Z(t) \ \forall t$ and the marginal of $(X(t), Z(t))$ is $\mu^*$ for all $t$. This implies the first part of the second claim. The second part of the second claim follows from Corollary 2.1 of \cite{Bhatt}, see also Corollary 6.3.7, p.\ 230, of \cite{Ari}.
\end{proof}

\begin{remark} Note that the control being a stationary relaxed Markov process does not imply that the state process itself is time-homogeneous Markov, or even Markov, in this case. This is due to nonuniqueness of solutions: for some $t > 0$, we can take some solution up to $t$ and then choose the regular conditional law of the process after $t$ from the set of solution measures initiated at $X(t)$ as a non-trivial function of the entire trajectory up to time $t$. This yields a non-Markov solution. More generally, one can follow the recipe of Section 12.3 of \cite{SV} to generate non-Markov solutions. \end{remark}

\end{document}